\documentclass [14 points] {article}
\usepackage{amssymb, amscd}
\usepackage{eucal}
\usepackage{amsmath}
\usepackage{amsthm}
\usepackage{graphicx}

\newcommand{\R}{{\mathbb  R}}
\def \t{\triangledown}
\numberwithin{equation}{section}
\newtheorem{thm}{\bf Theorem}[section]
\newtheorem{lem}[thm]{\bf Lemma}
\newtheorem{prop}[thm]{\bf Proposition}

\newtheorem{defn}[thm]{\bf Definition}
\newtheorem{corollary}[thm]{Corollary}
\theoremstyle{remark}
\newtheorem{rem}[thm]{\bf Remark}

\begin{document}

\title{\large\bf ASYMPTOTIC STABILITY OF DISSIPATED HAMILTON-POISSON SYSTEMS}
\author{Petre Birtea and Dan Com\u{a}nescu}
\date{ }
\maketitle

\begin{abstract}
We will further develop the study of the dissipation for a
Hamilton-Poisson system introduced in \cite{2}. We will give a
tensorial form of this dissipation and show that it preserves the
Hamiltonian function but not the Poisson geometry of the initial
Hamilton-Poisson system. We will give precise results about
asymptotic stabilizability of the stable equilibria of the initial
Hamilton-Poisson system.
\end{abstract}



\section{Introduction}

Let
\begin{equation*}
\dot{x}=\Pi\triangledown H(x)
\end{equation*}
be a Hamilton-Poisson system. If $C\in
C^{\infty}(\mathbb{R}^n,\mathbb{R})$ is a Casimir function then we
have the following obvious equality $\Pi\triangledown C=0$.

We will add to the Hamilton-Poisson system that we will denote by
$\xi_{\Pi}$ a dissipation term of the form $G\triangledown C$,
where $G$ is a certain symmetric matrix that will be discussed
below and $C$ is a Casimir function. The dissipated system $\xi$
will be
\begin{equation}\label{sistem}
\dot{x}=\Pi\triangledown H(x)+G\triangledown C(x)
\end{equation}

The notion of dissipative bracket was introduced by A. Kaufman
\cite{4} in his study of dissipative Hamilton-Poisson systems. In
\cite{7}, P.J. Morrison introduced the notion of metriplectic
systems which are Hamilton-Poisson systems with a dissipation of
metric type. Metric type dissipation was also introduced in \cite{7}
and is given by a dissipative bracket constructed from a metric
defined on the phase space. Dissipative terms and their implications
for dynamics have been studied in connection with various dynamical
systems derived from mathematical physics, see \cite{3}, \cite{grm},
\cite{5},\cite{5ad}, \cite{6}, \cite{7ad}.

The dissipation that we will study is the one introduced in
\cite{2}. Here we will write it in a tensorial form which enables
us to better understand its geometry. It is a particular form of
the metric dissipation introduced in \cite{7}. More precisely, our
system (\ref{sistem}) is the same system with the one described by
equations (25) in \cite{7} under the condition $(g^{ij})\t H=0$.
Like the dissipation from \cite{1}, our dissipation conserves the
energy. A dissipation form that conserves the symplectic leaves
and dissipates the Hamiltonian function have been studied in
\cite{bkmr}.

We will also illustrate the improvement of the stability result
obtained in \cite{2}. This improvement is obtained by the
observation that the dissipation conserves the energy.

\bigskip

\section{$G=\triangledown H\otimes\triangledown H-||\triangledown H||^2\mathbb{I}$}

\bigskip

In this section we introduce the dissipation matrix
$$G=\triangledown H\otimes\triangledown H-||\triangledown
H||^2\mathbb{I}.$$

This is the tensorial form of the dissipation constructed in
\cite{2} where the matrix $\mathbb{I}$ is the identity matrix on
$\mathbb{R}^n$. We denote by $\xi$ the dissipated system
\eqref{sistem} in the case when the matrix $G$ is as above.

\begin{lem}\label{1}We have the following properties for $G$:
\begin{itemize}
   \item [(i)] $G\t H=0$;
   \item [(ii)] $\t C\cdot G\t C\leq 0$, for any Casimir function $C$.
\end{itemize}
The equality holds in a point $x\in \mathbb{R}^n$ iff $\t C(x)$
and $\t H(x)$ are linear dependent.
\end{lem}

\begin{proof}For (i) we have
\begin{eqnarray*}
G\t H&=&(\t H\otimes \t H)\t H-||\t H||^2 \mathbb{I}\t H\\&=&(\t
H\cdot \t H)\t H-||\t H||^2\t H=0
\end{eqnarray*}

Analogous for (ii) we have
\begin{eqnarray*}
\t C\cdot G\t C&=&\t C\cdot [(\t H\otimes \t H)\t C-||\t H||^2
\mathbb{I}\t C]\\&=& \t C\cdot [(\t C\cdot \t H)\t H-||\t H||^2\t
C]\\&=&(\t C\cdot \t H)^2-||\t C||^2\cdot ||\t H||^2\leq0.
\end{eqnarray*}
by C-B-S inequality. Equality holds iff $\t H$ and $\t C$ are
linear dependent.
\end{proof}

For the perturbed system $\xi$ the initial Hamiltonian remains a
conservation low but the Casimir function C does not.

\begin{lem}\label{2}We have the following behavior:
\begin{itemize}
   \item [(i)] $H$ is conserved along the solutions of $\xi$.
   \item [(ii)] $C$ decreases along the solutions of $\xi$.
\end{itemize}
\end{lem}

\begin{proof} For (i) we have
\begin{eqnarray*}
\frac{d}{dt}H&=&\dot{x}\cdot \t H=(\Pi\t H+G\t C)\cdot \t H\\&=&
\Pi\t H\cdot \t H+[(\t H\otimes \t H)\t C-||\t H||^2\t C]\cdot \t
H\\&=&-\t H\cdot \Pi\t H=0
\end{eqnarray*}
since $\Pi$ is an antisymmetric matrix.

We do the same type of computation for (ii)
\begin{eqnarray*}
\frac{d}{dt}C&=&\dot{x}\cdot \t C=(\Pi\t H+G\t C)\cdot \t
C\\&=&-\t H\cdot \Pi \t C+\t C\cdot G\t C\\&=&\t C\cdot G\t C
\leq0
\end{eqnarray*}
by Lemma \ref{1} (ii).
\end{proof}

Next, we will study the relation between the set of equilibria
$E_{\xi_{\Pi}}$ for the unperturbed system $\xi_{\Pi}$ and the set
of equilibria $E_{\xi}$ for the perturbed system $\xi$.

\begin{prop}\label{dependent1}We have:
\begin{itemize}
 \item [(i)] If $\t C(x_e)\neq 0$ for a point $x_e\in \mathbb{R}^n$
then $x_e\in E_{\xi}\Leftrightarrow \t H(x_e)\, \texttt{and}\, \t
C(x_e)$ are linear dependent.
 \item [(ii)] $E_{\xi}\subset
E_{\xi_{\Pi}}$.
\end{itemize}
\end{prop}

\begin{proof} By definition we have that
$$x_e\in E_{\xi_{\Pi}}\Leftrightarrow \Pi\t H(x_e)=0$$
and
$$x_e\in E_{\xi}\Leftrightarrow \Pi\t H(x_e)+G\t C(x_e)=0$$

(i) Let $x_e\in E_{\xi}$. If we multiply both sides of the above
equality with $\t C(x_e)$ then
$$\t C(x_e)\cdot \Pi \t H(x_e)+\t C(x_e)\cdot G\t C(x_e)=0$$
$$\Leftrightarrow -\Pi\t C(x_e)\cdot \t H(x_e)+\t C(x_e)\cdot G\t C(x_e)=0\Leftrightarrow
\t C(x_e)\cdot G\t C(x_e)=0$$ which by Lemma \ref {1} (ii) implies
that $\t H(x_e)$ and $\t C(x_e)$ are linear dependent.

For the converse, if $\t H(x_e)$ and $\t C(x_e)$ are linear
dependent then there exists $\lambda\in \mathbb{R}$ such that $\t
H(x_e)=\lambda\t C(x_e)$. Consequently,
$$\Pi \t H(x_e)=\lambda\Pi \t C(x_e)=0.$$

In the case $\t H(x_e)=0$ we obtain that $G(x_e)=0$.

In the case $\t H(x_e)\neq 0$ we obtain that $\t
C(x_e)=\frac{1}{\lambda}\t H(x_e)$ and using Lemma \ref{1} (i) we
have $G\t C(x_e)=\frac{1}{\lambda}G\t H(x_e)=0$. In both cases we
conclude that $x_e\in E_{\xi}$.

\medskip

(ii) Let $x_e\in E_{\xi}$. If $\t C(x_e)=0$ then $G\t C(x_e)=0$
and we obtain that $\Pi \t H(x_e)=0$ which implies $x_e\in
E_{\xi_{\Pi}}$.

If $\t C(x_e)\neq 0$ then there exists $\lambda\in \mathbb{R}$
such that $\t H(x_e)=\lambda\t C(x_e)$.  We observe that $\Pi \t
H(x_e)=\lambda\Pi \t C(x_e)=0$ and consequently, $x_e\in
E_{\xi_{\Pi}}$.
\end{proof}

\section{Asymptotic stability}

We will briefly recall some definitions of stability for a
dynamical system on $\R^n$ that will be used later
\begin{equation}\label{egen}
    \dot x = f(x),
\end{equation} where $f \in C^\infty (\R^n, \R^n)$.

\begin{defn}\label{dstab}
An equilibrium point $x_e$ is stable if for any small neighborhood
$U$ of $x_e$ there is a neighborhood $V$ of $x_e$, $V \subset U$
such that if initially $x_0$ is in $V$, then $\phi(t,x_0) \in U$
for all $t>0$. If in addition we have
\begin{equation*}
    \lim_{t \to \infty} \phi(t,x_0) = x_e
\end{equation*}
then $x_e$ is called asymptotically stable.
\end{defn}

For studying more complicated asymptotic behavior we need to
introduce the notion of $\omega$-limit set. Let $\phi^t$ be the
flow defined by equation (\ref{egen}). The $\omega$-limit set of
$x$ is $\omega(x):=\{y\in \R^n|\exists t_1, t_2...\rightarrow
\infty $ {s.t.} $ \phi(t_k, x)\rightarrow y $ {as} $
k\rightarrow\infty\}$. The $\omega$-limit sets have the following
properties that we will use later. For more details, see
\cite{robinson}.

\begin{itemize}
  \item [(i)] If $\phi(t,x)=y$ for some $t\in \R$, then
  $\omega(x)=\omega(y)$.
  \item [(ii)] $\omega(x)$ is a closed subset and both positively and negatively
  invariant (contains complete orbits).
\end{itemize}

The above properties of $\omega$-limit sets have been used in the
proof of the following version of LaSalle theorem, see \cite{2}.

\begin{thm} \label{LaSalle}Let $x_e$ be an equilibrium point of a dynamical system
$$\dot{x}=f(x)$$
and $U$ a small compact neighborhood of $x_e$. Suppose there
exists $L:U\rightarrow \mathbb{R}$ a $C^1$ differentiable function
with $L(x)>0$ for all $x\in U\setminus\{x_e\}$, $L(x_e)=0$ and
$\dot{L}(x)\leq 0$. Let $E$ be the set of all points in $U$ where
$\dot{L}(x)= 0$. Let $M$ be the largest invariant set in $E$. Then
there exists a small neighborhood $V$ of $x_e$ with $V\subset U$
such that $\omega (x)\subset M$ for all $x\in V$.
\end{thm}

The next result describes the asymptotic behavior for the
solutions of the dissipated system $\xi$. We introduce the
following set $C_*=\{x\in \mathbb{R}^n\,|\,\t C(x)=0\}$.

\begin{thm}\label{asy} Let $x_e\in E_{\xi}$ be an equilibrium
point for the dissipated system $\xi$. Suppose there exists a
function $\psi(H,C)\in C^{\infty}(\mathbb{R}^2,\mathbb{R})$ such
that $\frac{\partial \psi}{\partial C}(H(x_e),C(x_e))>0$ and $x_e$
is a strict relative minimum for
$L(x)=\psi(H(x),C(x))-\psi(H(x_e),C(x_e))$. Then there exists a
small compact neighborhood $U$ of $x_e$ and another neighborhood
$V$ of $x_e$ with $V\subset U$ such that every solution of $\xi$
starting in $V$ approaches the largest invariant set $M$ in
$U\bigcap (E_{\xi}\bigcup C_*)$ as $t\rightarrow \infty$ ($M$ is
an attracting set).
\end{thm}

\begin{proof} We have $L(x_e)=0$ with $x_e$ being a strict
local minimum for $L$. Then there exists a small compact
neighborhood $U$ of $x_e$ such that $L(x)>0$ and $\frac{\partial
\psi}{\partial C}(H(x),C(x))>0$ on this neighborhood. We have the
following computation

\begin{eqnarray*}
\dot{L}(x)&=&\frac{\partial \psi}{\partial H}\t H\cdot
\dot{x}+\frac{\partial \psi}{\partial C}\t C\cdot
\dot{x}=(\frac{\partial \psi}{\partial H}\t H+\frac{\partial
\psi}{\partial C}\t C)(\Pi\t H+G\t C)\\&=&\frac{\partial
\psi}{\partial H}\t H \cdot \Pi\t H+\frac{\partial \psi}{\partial
H}\t H \cdot G\t C+\frac{\partial \psi}{\partial C}\t C \cdot
\Pi\t H+\frac{\partial \psi}{\partial C}\t C \cdot G\t
C\\&=&\frac{\partial \psi}{\partial H}\t H \cdot \Pi\t
H+\frac{\partial \psi}{\partial H}G\t H \cdot \t C-\frac{\partial
\psi}{\partial C}\Pi\t C \cdot \t H+\frac{\partial \psi}{\partial
C}\t C \cdot G\t C
\end{eqnarray*}

Using Lemma \ref{1} (i), antisymmetry of $\Pi$ and the fact that
$C$ is a Casimir function for $\Pi$ we obtain that
$$\dot{L}(x)=\frac{\partial \psi}{\partial C}\t C \cdot G\t C$$
From the hypothesis $\frac{\partial \psi}{\partial C}>0$ on $U$
and by Lemma \ref{1} (ii) we have
$$\dot{L}(x)\leq 0.$$
Using again Lemma \ref{1} (ii) we obtain that $E$ from the Theorem
\ref{LaSalle} equals $E_{\xi}\cup C_*$ which give us the result.
\end{proof}

\begin{rem} Observe that $L$ is also a Lyapunov function for the
unperturbed system $\xi_{\Pi}$. Consequently, by adding the
dissipation we render the stable points of $\xi_{\Pi}$ into
asymptotically stable points for $\xi$. Also the Casimir function
$C$ can be considered as a Lyapunov function.
\end{rem}

\begin{corollary}\label{cor}Let $x_e\in E_{\xi}$ and a function $\psi(H,C)\in
C^{\infty}(\mathbb{R}^2,\mathbb{R})$ such that $\frac{\partial
\psi}{\partial C}(H(x_e),C(x_e))>0$. Suppose that the function
$L(x)=\psi(H(x),C(x))-\psi(H(x_e),C(x_e))$ has the properties:

i) $\delta L(x_e)=0$;

ii) $\delta^2 L(x_e)$ \emph{is positive definite.}

Then there exists a small compact neighborhood $U$ of $x_e$ and
another neighborhood $V$ of $x_e$ with $V\subset U$ such that
every solution of $\xi$ starting in $V$ approaches the largest
invariant set $M$ in $U\bigcap (E_{\xi}\bigcup C_*)$ as
$t\rightarrow \infty$. Moreover $H$ remains constant along these
solutions and $C$ decreases along these solutions.
\end{corollary}

\begin{proof} It is easy to observe that the point $x_e$ is a strict
relative minimum of $L$. All the condition of the Theorem
\ref{asy} are satisfied and we obtain the desired result.
\end{proof}

\begin{rem} If we consider, in the Corolarry \ref{cor}, the function
$\psi(H,C)=H+C$ we obtain the stability result of \cite{2}.
\end{rem}

\begin{corollary}\label{corolar} In the hypotheses of the Theorem \ref{asy} we consider a point
$x_0\in V$. Suppose that the set $H^{-1}(\{h\})\bigcap U\bigcap
(E_{\xi}\bigcup C_*)$ has a unique point $x_h$ ($h=H(x_0)$). The
solution $x(t,x_0)$ of the system $\xi$ which verifies the initial
condition $x(0,x_0)=x_0$ satisfies the property
$$\lim_{t\rightarrow \infty}x(t,x_0)=x_h.$$
\end{corollary}

We can interpret the above result as follows, using the Lyapunov
function of Theorem \ref{asy} we obtain that every equilibrium
point in the neighborhood $V$ is asymptotically stable for the
dynamical system on the corresponding level set.

\section{Applications to the rigid body dynamics}

The motion of a rigid body can be reduced to the translation of
center of mass and rotation about it. Rotation is conveniently
described, in a coordinate system with the origin at the center of
mass and the axes along principal central axes of inertia, by
Euler's equations. This equations can be written in the following
form
$$\left\{%
\begin{array}{ll}
    \dot{x_1}=(\frac{1}{I_3}-\frac{1}{I_2})x_2x_3+u_1 \\
    \dot{x_2}=(\frac{1}{I_1}-\frac{1}{I_3})x_1x_3+u_2\\
    \dot{x_3}=(\frac{1}{I_2}-\frac{1}{I_1})x_1x_2+u_3\\
\end{array}%
\right.$$ where
$x_1=I_1\omega_1,\,x_2=I_2\omega_2,\,x_3=I_3\omega_3$ are the
components of $\bf{x}$, $I_1,I_2,I_3$ are the principal moments of
inertia, $\omega_1,\omega_2,\omega_3$ are the components of the
angular velocity and $u_1,u_2,u_3$ are the components of applied
torques $\bf{u}$. In this paper we suppose that $I_1>I_2>I_3$.

\hspace{0.5cm} The system of free rotations, denoted by
$\xi_{\Pi}$,  has the property that $\bf{u}=\bf{0}$. It has the
following Hamilton-Poisson realization
$$((so(3))^*\approx \mathbb{R}^3,\, \{\cdot\,,\cdot\}_{-},\,H)$$
where $\{\cdot\,,\cdot\}_{-}$ is minus-Lie-Poisson structure on
$(so(3))^*\approx \mathbb{R}^3$ generated by the matrix
$$\Pi_{-}=\left(%
\begin{array}{ccc}
  0 & -x_3 & x_2 \\
  x_3 & 0 & -x_1 \\
  -x_2 & x_1 & 0 \\
\end{array}%
\right)$$ and the Hamiltonian $H$ is given by
$$H(x_1,x_2,x_3)=\frac{1}{2}(\frac{x_1^2}{I_1}+\frac{x_2^2}{I_2}+\frac{x_3^2}{I_3}).$$
It is easy to see that the function $C_0\in
C^{\infty}(\mathbb{R}^3,\mathbb{R})$ given by
$$C_0(x_1,x_2,x_3)=\frac{1}{2}(x_1^2+x_2^2+x_3^2)$$ is a Casimir of our configuration
$((so(3))^*\approx \mathbb{R}^3,\,
\{\cdot\,,\cdot\}_{-})$, called the \emph{standard Casimir},  i.e.
$$\{C_0,f\}=0$$
for each $f\in C^{\infty}(\mathbb{R}^3,\mathbb{R})$ where
$$\{f,g\}=(\t f)^T\Pi_{-}\t g.$$
The set of the Casimir functions is given by
$$\{C=\varphi (C_0)\,|\, \varphi\in C^{\infty}(\mathbb{R},\mathbb{R})\}.$$
It is known that the set of equilibria is given by
$$E_{\xi_{\Pi}}=\{(M_0,0,0)\,|\,M_0\in \mathbb{R}\}\bigcup\{(0,M_0,0)\,|\,M_0\in
\mathbb{R}\}\bigcup \{(0,0,M_0)\,|\,M_0\in \mathbb{R}\}$$
and the equilibrium points of the form $(M_0,0,0)$ or $(0,0,M_0)$
are stable and the equilibrium points of the form $(0,M_0,0)$ with
$\,M_0\neq 0$ are not stable.

In this paper we consider a torque of the form
$$\textbf{u}=G\t C,$$
with $C$ a Casimir function and $G$ the matrix
$$G=\triangledown H\otimes\triangledown H-||\triangledown H||^2\mathbb{I}=\left(%
\begin{array}{ccc}
  -\frac{x_2^2}{I_2^2}-\frac{x_3^2}{I_3^2} & \frac{x_1 x_2}{I_1 I_2} & \frac{x_1 x_3}{I_1 I_3} \\
  \frac{x_1 x_2}{I_1 I_2} & -\frac{x_1^2}{I_1^2}-\frac{x_3^2}{I_3^2} & \frac{x_2 x_3}{I_2 I_3} \\
  \frac{x_1 x_3}{I_1 I_3} & \frac{x_2 x_3}{I_2 I_3} & -\frac{x_1^2}{I_1^2}-\frac{x_2^2}{I_2^2} \\
\end{array}%
\right).$$

For the case of the free rigid body the above matrix is the one
also used in \cite{7}.

\bigskip

{\bf Case I.} We first consider the standard Casimir $C_0$ in
order to construct the perturbation. We obtain a torque
$\textbf{u}_0$ with the components:
$$\left\{%
\begin{array}{ll}
   u^0_1=x_1[(\frac{1}{I_1}-\frac{1}{I_2})\frac{x_2^2}{I_2}+(\frac{1}{I_1}-\frac{1}{I_3})\frac{x_3^2}{I_3}] \\
   u^0_2=x_2[(\frac{1}{I_2}-\frac{1}{I_1})\frac{x_1^2}{I_1}+(\frac{1}{I_2}-\frac{1}{I_3})\frac{x_3^2}{I_3}]\\
   u^0_3=x_3[(\frac{1}{I_3}-\frac{1}{I_1})\frac{x_1^2}{I_1}+(\frac{1}{I_3}-\frac{1}{I_2})\frac{x_2^2}{I_2}]\\
\end{array}%
\right.$$

Let $\xi^0$ be the rotation system with the torque $\textbf{u}_0$.
It is easy to see that the set of the  equilibrium points of the
dissipated system is $$E_{\xi^0}=E_{\xi_{\Pi}}.$$ Also we have
$$C_*=\{(0,0,0)\}.$$

We consider $M_0\in \mathbb{R}^*$. The point $x_e=(0,0,M_0)$ is an
equilibrium point of the rotation system $\xi^0$. We define the
function $\psi :\mathbb{R}^2\rightarrow \mathbb{R}$ by
$$\psi (H,C_0)=(C_0-\frac{M_0^2}{2})^2+C_0-I_3H.$$
This function has the properties:
$$\frac{\partial \psi}{\partial
C_0}(H,C_0)=2(C_0-\frac{M_0^2}{2})+1$$ and
$$\frac{\partial \psi}{\partial C_0}(H(x_e),C_0(x_e))=1>0.$$
Using our notations we introduce the Lyapunov function
\begin{eqnarray*}
L(x_1,x_2,x_3)&=&(C_0(x_1,x_2,x_3)-\frac{M_0^2}{2})^2+C_0(x_1,x_2,x_3)-I_3H(x_1,x_2,x_3)
\\&=&(\frac{1}{2}(x_1^2+x_2^2+x_3^2)-\frac{M_0^2}{2})^2+\frac{1}{2}
(x_1^2+x_2^2+x_3^2)-\frac{I_3}{2}(\frac{x_1^2}{I_1}+\frac{x_2^2}{I_2}+\frac{x_3^2}{I_3}).
\end{eqnarray*}

It is easy to see that
$$\delta L(x_e)=0$$ and
$$\delta^2 L(x_e)=\left(%
\begin{array}{ccc}
  1-\frac{I_3}{I_1} & 0 & 0 \\
  0 & 1-\frac{I_3}{I_2} & 0 \\
  0 & 0 & 2M_0^2 \\
\end{array}%
\right) .$$ Using our hypotheses we observe that $\delta^2 L(x_e)$
is positive definite. Consequently, the hypotheses of the
Corollary \ref{cor} are satisfied and we have the following
stability result.

\begin{thm}There exists a small compact neighborhood $U$ of $x_e$ and an
other $V\subset U$ such that every solution of $\xi^0$ starting in
$V$ approaches $U\cap \{(0,0,x)\,|\,x\in \mathbb{R}\}$ as
$t\rightarrow \infty $. Moreover $H$ remains constant along
solutions and $C_0$ decreases along this solutions.\end{thm}

If $U$ is sufficiently small then $U\cap E_{\xi^0}\subset
\{(0,0,x)\,|\,sgn(x)=constant\}$. Suppose that $sgn(x)=1$. Using
the Corollary \ref{corolar} we have the following asymptotic
stability result.

\begin{thm}If $x_0\in V$ then the solution $x(t,x_0)$ of the system
$\xi_0$ which verifies the initial condition $x(0,x_0)=x_0$
satisfies the property
$$\lim_{t\rightarrow \infty}x(t,x_0)=(0,0,\sqrt{2I_3H(x_0)}).$$
\end{thm}

We present the simulation of the rotation of the rigid body in the
case when the principal moments of inertia are $I_1=4$, $I_2=1.5$,
$I_3=1$ and the  initial conditions are $(-0.1,\,0.2,\,0.175)$.
\begin{figure}[h]
\begin{center}

  \includegraphics[width=7cm]{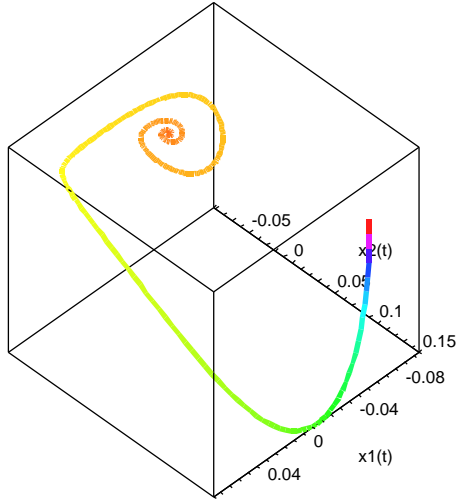}\\
  \caption{The rotation of the rigid body}\label{*}
\end{center}
\end{figure}

\begin{rem} We see that the function $C_1=C_0+(C_0-\frac{M_0^2}{2})^2$
is also a Casimir function. If we consider
$\psi_1(H,C_1)=C_1-I_3H$ we have $\frac{\partial \psi_1}{\partial
C_1}=1>0$ and $\psi (H,C_0)=\psi_1(H,C_1)$. It is possible to
apply Corollary \ref{cor} for the system $\xi_1$ of the form
(\ref{sistem}) with the Casimir function $C_1$ and the equilibrium
point $x_e=(0,0,M_0)$.
\end{rem}

\bigskip

{\bf Case II.} We study the case of the Casimir $C_2=-C_0$.We
obtain a torque $\textbf{u}_2$ with the components:
$$\left\{%
\begin{array}{ll}
   u^2_1=-x_1[(\frac{1}{I_1}-\frac{1}{I_2})\frac{x_2^2}{I_2}+(\frac{1}{I_1}-\frac{1}{I_3})\frac{x_3^2}{I_3}] \\
   u^2_2=-x_2[(\frac{1}{I_2}-\frac{1}{I_1})\frac{x_1^2}{I_1}+(\frac{1}{I_2}-\frac{1}{I_3})\frac{x_3^2}{I_3}]\\
   u^2_3=-x_3[(\frac{1}{I_3}-\frac{1}{I_1})\frac{x_1^2}{I_1}+(\frac{1}{I_3}-\frac{1}{I_2})\frac{x_2^2}{I_2}]\\
\end{array}%
\right.$$

Let $\xi^2$ the rotation system with the torque $\textbf{u}_2$. It
is easy to see that the set of the  equilibrium points of the
dissipated system is $$E_{\xi^2}=E_{\xi_{\Pi}}.$$ Also we have
$$C_{2*}=\{(0,0,0)\}.$$
We consider the equilibrium point $x_e=(M_0,0,0)$ with $M_0\in
\mathbb{R}^*$.

If we define
$$\psi_2 (H,C_2):=H+(C_2+\frac{M_0^2}{2})^2+\frac{C_2}{I_1},$$
we have that
$$\frac{\partial \psi_2}{\partial
C_2}(H,C_2)=2(C_2+\frac{M_0^2}{2})+\frac{1}{I_1}$$ and
$$\frac{\partial \psi_2}{\partial C_2}(H(x_e),C_2(x_e))=\frac{1}{I_1}>0.$$
In this situation we can apply Corollary \ref{cor} and we obtain
the stability result.

\begin{thm}There exists a small compact neighborhood $U$ of $x_e$ and an
other $V\subset U$ such that every solution of $\xi^2$ starting in
$V$ approaches $U\cap \{(x,0,0)\,|\,x\in \mathbb{R}\}$ as
$t\rightarrow \infty $. Moreover $H$ remains constant along
solutions and $C_2=-C_0$ decreases along this solutions.\end{thm}

If $U$ is sufficiently small then $U\cap E_{\xi^2}\subset
\{(x,0,0)\,|\,sgn(x)=constant\}$. Suppose that $sgn(x)=1$. Using
the Corollary \ref{corolar} we obtain the asymptotic stability
result.

\begin{thm}If $x_0\in V$ then the solution $x(t,x_0)$ of the system $\xi_2$
which verifies the initial condition $x(0,x_0)=x_0$ satisfies the
property
$$\lim_{t\rightarrow \infty}x(t,x_0)=(\sqrt{2I_1H(x_0)},0,0).$$
\end{thm}

{\bf Case III.} In \cite{2} were introduced the functions
$$C_3=(C_0-\frac{M_0^2}{2})^2-\frac{C_0}{I_1}=(C_2+\frac{M_0^2}{2})^2+\frac{C_2}{I_1}$$
and
$$L_3(x)=H(x)+C_3(x)-H(x_e)-C_3(x_e).$$
It is easy to see that:

i) $\delta L_3(x_e)=0$;

ii) $\delta^2 L_3(x_e)$ is positive definite because the Hessian matrix of $L_3$ in $x_e$ is $$\left(%
\begin{array}{ccc}
  8M_0^2 & 0 & 0 \\
  0 & \frac{1}{I_2}-\frac{1}{I_1} & 0 \\
  0 & 0 & \frac{1}{I_3}-\frac{1}{I_1} \\
\end{array}%
\right).$$

We denote by $\xi_3$ the system of the form (\ref{sistem}) with
the Casimir $C_3$. In \cite{2} has been proved the following
result.

\begin{thm}There exists a small compact neighborhood $U$ of $x_e$
and an other $V\subset U$ such that every solution of $\xi^3$
starting in $V$ approaches $U\cap \{(x,0,0)\,|\,x\in \mathbb{R}\}$
as $t\rightarrow \infty $.\end{thm}

Our results applies to equilibria that are already stable. By adding
the dissipation we can make them asymptotically stable if certain
conditions are satisfied. Going from asymptotic stability on the
level set to stability in the whole space have been studied in
\cite{aese}. By adding other type of controls one can stabilize the
unstable equilibria. This subject is studied for example in
\cite{aey}, \cite{bkms}, \cite{blma}, \cite{sosu}.

\end{document}